\documentclass{article}
\usepackage{amsfonts,latexsym,hyperref}
\usepackage{multirow}
\newcounter{satznum}
\newtheorem{theorem}{Theorem}[satznum]

\newtheorem{lemma}[theorem]{Lemma}

\newtheorem{proposition}[theorem]{Proposition}
\newenvironment{remark}
 {\begin{trivlist}\item[]{\bf Remark.}}
 {\end{trivlist}}
\newenvironment{remarks}
 {\begin{trivlist}\item[]{\bf Remarks.}}
 {\end{trivlist}}
\newenvironment{example}
 {\begin{trivlist}\item[]{\bf Example.}}
 {\end{trivlist}}

\newenvironment{proof}
 {\begin{trivlist}\item[]{\bf Proof.}}
 {\end{trivlist}}
{\makeatletter
\gdef\me{{\mathbb E}} 
\gdef\nz{{\mathbb N}} 
\gdef\pr{{\mathbb P}} 
\gdef\rz{{\mathbb R}} 
\gdef\gz{{\mathbb Z}} 
}
\newcounter{todocounter}

\makeatletter
\def\@MRExtract#1 #2!{#1}
\newcommand{\MR}[1]{
  \xdef\@MRSTRIP{\@MRExtract#1 !}
  \href{http://www.ams.org/mathscinet-getitem?mr=\@MRSTRIP}{MR\@MRSTRIP}}
\makeatother

\begin{document}
   \section*{ABSORPTION TIME AND TREE LENGTH OF THE KINGMAN COALESCENT AND
   THE GUMBEL DISTRIBUTION}
   {\sc M.~M\"ohle}\footnote{Mathematisches Institut, Eberhard Karls Universit\"at T\"ubingen,
   Auf der Morgenstelle 10, 72076 T\"ubingen, Germany, E-mail address: martin.moehle@uni-tuebingen.de}
   and
   {\sc H.~Pitters}\footnote{Department of Statistics, University of Oxford,
   1 South Parks Road, Oxford OX1 3TG, UK,
   E-mail address: helmut.pitters@stats.ox.ac.uk}
\begin{center}
   September 17, 2014
\end{center}
\begin{abstract}
   Formulas are provided for the cumulants and the moments of the
   time $T$ back to the most recent common ancestor of the Kingman
   coalescent. It is shown that both the $j$th cumulant and the $j$th
   moment of $T$ are linear combinations of the values $\zeta(2m)$,
   $m\in\{0,\ldots,\lfloor j/2\rfloor\}$, of the Riemann zeta function
   $\zeta$ with integer coefficients. The proof is based on a solution
   of a two-dimensional recursion with countably many initial values.
   A closely related strong convergence result for the tree length
   $L_n$ of the Kingman coalescent restricted to a sample of size $n$
   is derived. The results give reason to revisit the moments and
   central moments of the classical Gumbel distribution.

   \vspace{2mm}

   \noindent Keywords: absorption time; cumulants; Euler--Mascheroni integrals;
   Gumbel distribution; infinite convolution; 
   Kingman coalescent; moments; most recent common ancestor;
   tree length; zeta function

   \vspace{2mm}

   \noindent Running head: Absorption time and tree length of the
   Kingman coalescent

   \vspace{2mm}

   \noindent 2010 Mathematics Subject Classification:
            Primary 60C05;   
                    60J28    
            Secondary 60G50; 
            92D25            
\end{abstract}
\subsection{Introduction} \label{intro}
\setcounter{theorem}{0}
   Kingman's coalescent \cite{kingman1, kingman2}, the most important
   coalescent among the class of all exchangeable coalescents, is a
   continuous time Markov process $\Pi=(\Pi_t)_{t\ge 0}$ with state space
   ${\cal P}$, the set of partitions of $\nz:=\{1,2,\ldots\}$. If this
   process is in a state with $b$ blocks, then by definition during each
   transition any two blocks merge together at rate $1$. This process starts
   in the partition of $\nz$ into singletons and reaches its absorbing state,
   the partition consisting of the single block $\nz$, in finite time almost
   surely. Recently there is much interest in certain functionals of
   coalescent processes (restricted to a sample of size $n\in\nz$) such as
   the number of jumps, the absorption time, or the total tree length to
   mention a few of them. In this manuscript we provide some new results on
   the absorption time and the total tree length of the Kingman coalescent.
   These results give reason to revisit the classical Gumbel distribution.

   The article is organized as follows. In the following Section
   \ref{functionals} results on functionals of the Kingman
   coalescent, such as the absorption time and the tree length, are
   provided. Section \ref{gumbel} is devoted to the classical Gumbel
   distribution. We recall well known results of the Gumbel
   distribution but also shed some new light in particular on the
   central moments of this distribution. Proofs are provided in
   Section \ref{proofs}. The article finishes with an appendix where
   the first central moments of the Gumbel distribution are given
   explicitly. The appendix furthermore provides the spectral
   decomposition of the transition matrix of a pure death process
   having distinct death rates.
\subsection{Absorption time and tree length} \label{functionals}
\setcounter{theorem}{0}
   Kingman \cite{kingman1} studied the absorption time $T$ 
   of $\Pi$. In the biological
   context $T$ is called the time back to the most recent common ancestor or
   the age of the most recent common ancestor.
   It is well known that $T=\sum_{k=2}^\infty
   \tau_k$ is an infinite convolution of independent exponentially
   distributed random variables $\tau_k$ with parameter
   $\lambda_k:=k(k-1)/2$. Using the inversion formula for
   Fourier transforms it is readily checked
   that $T$ has a bounded and infinitely often differentiable density
   $g:\rz\to [0,\infty)$ with respect
   to Lebesgue measure $\lambda$ on $\rz$.
   Kingman \cite[p.~37, Eq.~(5.9)]{kingman1} showed that
   $g(t)=\sum_{k=2}^\infty (-1)^k(2k-1)\lambda_k e^{-\lambda_k t}$,
   $t\in (0,\infty)$. It is furthermore known
   (Watterson \cite[p.~213]{watterson}, Tavar\'e \cite[p.~132]{tavare1}
   that $T$ has mean $\me(T)=2$ and variance ${\rm Var}(T)=4\pi^2/3-12\approx 1.15947$
   and, hence, second moment $\me(T^2)=4\pi^2/3-8\approx 5.15947$.
   To the best of the authors knowledge 
   higher
   moments and cumulants of $T$ have not been derived so far.
   Proposition \ref{prop1} and Theorem \ref{thm2} below
   provide full information on the cumulants and moments of $T$.

\begin{proposition} \label{prop1}
   For the Kingman coalescent the absorption time $T$ has cumulants
   \begin{equation} \label{cumulant1}
      \kappa_j(T)
      \ =\ (j-1)!2^j\sum_{k=2}^\infty \frac{1}{k^j(k-1)^j},
      \qquad j\in\nz,
   \end{equation}
   and moments
   \begin{equation} \label{mean1}
   \me(T^j)\ =\ j!2^j\sum_{k=2}^\infty \frac{(-1)^k(2k-1)}{k^j(k-1)^j}
   \ =\ j!\sum_{m=1}^\infty \frac{1}{m^j}\bigg(\frac{4m-1}{(2m-1)^j}-\frac{4m+1}{(2m+1)^j}\bigg),
   \qquad j\in\nz.
   \end{equation}
   In particular, $\kappa_j(T)\sim (j-1)!$ and
   $\me(T^j)\sim 3j!$ as $j\to\infty$.
   Alternatively,
   \begin{equation} \label{meanalt}
      \me(T^j)\ =\
      j!\sum_{2\le k_1\le\cdots\le k_j}\frac{1}{\lambda_{k_1}\cdots\lambda_{k_j}}
      \ =\ j!2^j\sum_{2\le k_1\le\cdots\le k_j}\prod_{i=1}^j \frac{1}{k_i(k_i-1)},
      \qquad j\in\nz.
   \end{equation}
\end{proposition}
\begin{remarks}
   1. For $n\in\nz$ let $T_n$ denote the absorption time of the Kingman
   coalescent restricted to a sample of size $n$. The proof of Proposition
   \ref{prop1} provided in Section \ref{proofs} shows that $T_n\to T$
   almost surely as $n\to\infty$ with convergence of all moments, which
   implies (see, for example, \cite[Proposition 3.12]{kallenberg}) the
   convergence $T_n\to T$ in $L^p$ for any $p\in (0,\infty)$. Moreover,
   the sequence $(T_n^p)_{n\in\nz}$ is uniformly
   integrable for any $p\in (0,\infty)$.

   2. Let $g_n:\rz\to [0,\infty)$ denote the density of $T_n$. Using the
   inversion formula for Fourier transforms it is straightforward to
   establish the local convergence result
   $\sup_{t\in\rz}|g_n(t)-g(t)|=O(1/n)$. Scheff\'e's
   theorem (see, for example, \cite[Theorem 16.12]{billingsley} implies
   that $g_n\to g$ in $L^1$, 
   and, therefore,
   $|\pr(T_n\in B)-\pr(T\in B)|=|\int_B g_n\,{\rm d}\lambda-\int_B g\,{\rm d}\lambda|
   \le\int |g_n-g|\,{\rm d}\lambda\to 0$ for all Borel sets $B\subseteq\rz$.

   3. With some more effort it can be even verified that
   $\sup_{t\in\rz}|n(g_n(t)-g(t))-2g'(t)|=O(1/n)$. The proof of this result
   is again based on the inversion formula for Fourier transforms, however
   a bit technical and therefore omitted here. We will not use this advanced
   local convergence result in our further considerations.
\end{remarks}
The following result shows that the cumulants $\kappa_j(T)$
and the moments $\me(T^j)$ of $T$ are related
to the Riemann zeta function $\zeta$. More precisely, $\kappa_j(T)$ and
$\me(T^j)$ are both linear combinations of the zeta values $\zeta(2m)$,
$m\in\{0,\ldots,\lfloor j/2\rfloor\}$, with integer coefficients.
\begin{theorem} \label{thm2}
   For all $j\in\nz$,
   \begin{equation} \label{cumulant2}
      \kappa_j(T)\ =\ (-1)^j2^{j+1}\sum_{m=0}^{\lfloor j/2\rfloor}
        \frac{(2j-2m-1)!}{(j-2m)!}\zeta(2m)
   \end{equation}
   and
   \begin{equation} \label{mean2}
      \me(T^j)\ =\
      (-1)^jj2^{j+1}\sum_{m=0}^{\lfloor j/2\rfloor}
      (2m-1)\bigg(1-\frac{1}{2^{2m-1}}\bigg)\frac{(2j-2m-2)!}{(j-2m)!}
      \,\zeta(2m),
   \end{equation}
   where $\zeta$ denotes the zeta function.
\end{theorem}
\begin{remarks}
   1. The coefficient in (\ref{mean2}) in front of $\zeta(2m)$ is integer,
   since $2^{j+1}(1-1/2^{2m-1})=2^{j+1}-2^{j-2m+2}\in\gz$ and
   $(2j-2m-2)!/(j-2m)!\in\gz$ for all $j\in\nz$ and all $m\in\{0,\ldots,
   \lfloor j/2\rfloor\}$.

   2. Let $B_0,B_1,\ldots$ denote the Bernoulli numbers defined
   recursively via $B_0:=1$ and $B_n:=-1/(n+1)\sum_{k=0}^{n-1}
   {{n+1}\choose k}B_k$ for $n\in\nz$. For instance, $B_1=-1/2$,
   and $B_2=1/6$.
   Since $\zeta(2m)=(-1)^{m-1}(2\pi)^{2m}/(2(2m)!)B_{2m}$ is a rational
   multiple of $\pi^{2m}$, Theorem
   \ref{thm2} implies that $\kappa_j(T)$ and $\me(T^j)$
   are both polynomials
   of degree $\lfloor j/2\rfloor$ with rational coefficients evaluated
   at $\pi^2$. The following tables are easily computed using (\ref{cumulant2})
   and (\ref{mean2}).
   \begin{center}
      \begin{tabular}{|r|l|l|r|}
         \hline
         $j$ & $\kappa_j(T)$ in terms of $\zeta(.)$ & $\kappa_j(T)$ in terms of $\pi$ & $\kappa_j(T)$ numerically \\
         \hline
         $1$ & $2$ & $2$ & $2.00000$\\
         $2$ & $8\zeta(2)-12$ & $\frac{4}{3}\pi^2-12$ & $1.15947$\\
         $3$ & $160-96\zeta(2)$ & $160-16\pi^2$ & $2.08633$\\
         $4$ & $192\zeta(4)+1920\zeta(2)-3360$ & $\frac{32}{15}\pi^4+320\pi^2-3360$ & $6.07947$\\
         $5$ & $-7680\zeta(4)-53760\zeta(2)+96768$ & $-\frac{256}{3}\pi^4-8960\pi^2+96768$ & $24.10210$\\
         \hline
      \end{tabular}\\
      \ \\
      Table 1: The first five cumulants of $T$.
   \end{center}

   \vspace{2mm}

   \begin{center}
      \begin{tabular}{|r|l|l|r|}
         \hline
         $j$ & $\me(T^j)$ in terms of $\zeta(.)$ & $\me(T^j)$ in terms of $\pi$ & $\me(T^j)$ numerically \\
         \hline
         $1$ & $2$ & $2$ & $2.00000$\\
         $2$ & $8\zeta(2)-8$ & $\frac{4}{3}\pi^2-8$ & $5.15947$\\
         $3$ & $96-48\zeta(2)$ & $96-8\pi^2$ & $17.04317$\\
         $4$ & $672\zeta(4)+768\zeta(2)-1920$ & $\frac{112}{15}\pi^4+128\pi^2-1920$ & $70.63058$\\
         $5$ & $-20160\zeta(4)-19200\zeta(2)+53760$ & $-224\pi^4-3200\pi^2+53760$ & $357.62953$\\
         \hline
      \end{tabular}\\
      \ \\
      Table 2: The first five moments of $T$.
   \end{center}
\end{remarks}
   The proofs provided in Section \ref{proofs} rely on
   the fact that the jump chain of the block counting process of
   the Kingman coalescent is deterministic, which implies that $T$ is an
   infinite convolution of exponentially distributed random variables.
   The proof of Theorem \ref{thm2} is based on a solution of a
   two-dimensional recursion with an infinite number of initial values
   (see Lemma \ref{reclemma}).

   Our methods do not seem to be directly applicable to
   (absorption times of) other exchangeable coalescent processes,
   since the jump chain of the block counting process of
   a coalescent with multiple collisions is in general not
   deterministic.

   Our methods are partly useful to analyze further functionals of
   the Kingman $n$-coalescent (restricted to a sample of size $n\in\nz$).
   As an example we provide detailed information on the tree length
   $L_n$ (the sum of the lengths of all branches of the
   $n$-coalescent tree) of the Kingman $n$-coalescent.
   Let $G$ be a standard Gumbel
   distributed random variable with distribution function $x\mapsto\exp(-\exp(-x))$, $x\in\rz$.
   It is well known (see, for example, \cite[p.~22--23]{tavare2}, \cite[Lemma 7.1]{drmotaiksanovmoehleroesler}
   or \cite[Lemma 2.21]{etheridge})
   that $L_n$ has the same distribution as the maximum of $n-1$ independent
   and exponentially distributed random variables with parameter $1/2$
   and that $L_n/2-\log n\to G$ in distribution as $n\to\infty$. We verify
   the following stronger convergence result.
\begin{theorem}[Strong asymptotics of the tree length] \label{thm3}
   For the Kingman coalescent, as $n\to\infty$, $G_n:=L_n/2-\log n\to G$
   almost surely and in $L^p$ for any $p\in (0,\infty)$, where $G$ is
   standard Gumbel distributed. Moreover, the sequence $(G_n^p)_{n\in\nz}$
   is uniformly integrable for any $p\in (0,\infty)$.
\end{theorem}
\begin{remarks}
   1. Theorem \ref{thm3} implies the convergence $\kappa_j(G_n)\to\kappa_j(G)$
   of all cumulants and the convergence $\me(G_n^j)\to\me(G^j)$ of all
   moments, which is one of the starting points of the
   proof of Theorem \ref{thm3}. Formulas for the cumulants and the
   moments of $L_n$ are provided in (\ref{lncum}), (\ref{lnmom}) and (\ref{lnmom2}).

   2. For asymptotic results on the tree length $L_n$ for $\Lambda$-coalescents
   with $\Lambda=\beta(a,b)$ being the beta distribution with parameters
   $a,b\in (0,\infty)$ we refer the reader to
   \cite[Theorem 1]{kersting} for $0<a<1$ and $b:=2-a$ and to
   \cite[Corollary 4.3 and Theorem 5.2]{drmotaiksanovmoehleroesler} for $a=b=1$
   (Bolthausen--Sznitman coalescent). The asymptotics of
   $L_n$ for $\Xi$-coalescents with dust is provided in
   \cite[Theorem 3]{moehleproper}.
\end{remarks}
\subsection{The Gumbel distribution revisited} \label{gumbel}
\setcounter{theorem}{0}
Theorem \ref{thm3} and its proof smooth the way to establish a result
(Theorem \ref{cmgumbel}) on the Gumbel distribution, also called the
extreme value distribution of type $1$. Recall that a standard
Gumbel distributed random variable $G$ has
(see, for example, \cite[p.~12, Eqs.~(22.29) and (22.30)]{johnsonkotz})
cumulants $\kappa_1=\kappa_1(G)=\gamma$ (Euler's constant) and
$\kappa_j=\kappa_j(G)=(-1)^j\Psi^{(j-1)}(1)=(j-1)!\,\zeta(j)$, $j\ge 2$.
Note that $\kappa_j\sim (j-1)!$ as $j\to\infty$.
Moreover, $G$ has moments
\begin{equation} \label{gumbelmoments}
m_n\ :=\ \me(G^n)
\ =\ \int_0^\infty (-\log u)^n e^{-u}\,{\rm d}u
\ =\ (-1)^n \Gamma^{(n)}(1), \quad n\in\nz_0:=\{0,1,2,\ldots\}.
\end{equation}
The integral in (\ref{gumbelmoments}) is sometimes
called the $n$th Euler--Mascheroni integral. In the analytic
community its interpretation as the $n$th moment of the Gumbel
distribution often remains unmentioned. It is well known
that $m_n\sim n!$ as $n\to\infty$. The moments $m_1,m_2,\ldots$ can be
recursively computed via the relation between cumulants and moments
\begin{equation} \label{gumbelrec}
   m_n
   \ =\ \sum_{k=1}^n {{n-1}\choose{k-1}}\kappa_k m_{n-k}
   \ =\ \gamma m_{n-1}+(n-1)!\sum_{k=2}^n \frac{1}{(n-k)!}\zeta(k)m_{n-k},
   \quad n\in\nz,
\end{equation}
in agreement with the recursion provided on top of p.~214 in
the book of Boros and Moll \cite{borosmoll}.
A useful and well known formula for the moments is
\begin{equation} \label{mnexplicit}
m_n\ =\ \sum_{\pi\in{\cal P}_n}\prod_{B\in\pi}\kappa_{|B|},
\qquad n\in\nz,
\end{equation}
where the sum extends over all partitions $\pi$ of the set ${\cal
P}_n$ of all partitions of $\{1,\ldots,n\}$, and the product has
to be taken over all blocks $B$ of $\pi$. Since there exist
$n!/(a_1!\cdots a_n! 1!^{a_1}\cdots n!^{a_n})$ partitions
$\pi\in{\cal P}_n$ having $a_j$ blocks of size $j$, $1\le j\le n$,
the above formula can be also written as
   \begin{equation} \label{mnsum}
   m_n\ =\ n!\sum_{a_1,\ldots,a_n}
   \prod_{i=1}^n \frac{1}{a_i!}\bigg(\frac{\kappa_i}{i!}\bigg)^{a_i}
   \ =\ n!\sum_{a_1,\ldots,a_n}\frac{\gamma^{a_1}}{a_1!}
   \prod_{i=2}^n \frac{1}{a_i!}\bigg(\frac{\zeta(i)}{i}\bigg)^{a_i},
   \qquad n\in\nz,
   \end{equation}
where the sum $\sum_{a_1,\ldots,a_n}$ extends over all
$a_1,\ldots,a_n\in\nz_0$ satisfying $\sum_{i=1}^n ia_i=n$. For
instance,
$m_1=\gamma$, $m_2=\gamma^2+\zeta(2)$,
and $m_3=\gamma^3+3\gamma\zeta(2)+2\zeta(3)$.
In particular, $m_n$ is a polynomial of degree $n$ in the variable
$x:=(x_1,\ldots,x_n):=(\gamma,\zeta(2),\ldots,\zeta(n))$ with
nonnegative integer coefficients.
In the following we focus on the central moments
$m_n':=\me((G-\gamma)^n)$, $n\in\nz_0$, of the Gumbel distribution. Clearly,
\begin{equation} \label{gumbelmoments2}
   m_n\ =\ \me((G-\gamma+\gamma)^n)
   \ =\ \sum_{j=0}^n {n\choose j}\gamma^{n-j}m_j',\qquad n\in\nz_0,
\end{equation}
and
$$
m_n'\ =\ \sum_{j=0}^n {n\choose j}(-\gamma)^j m_{n-j}
\ =\ n!\sum_{j=0}^n \frac{(-\gamma)^j}{j!}\frac{m_{n-j}}{(n-j)!}
\ \sim\ n!\sum_{j=0}^\infty \frac{(-\gamma)^j}{j!}
\ =\ n!e^{-\gamma}
$$
as $n\to\infty$. As for the moments
we obtain for the central moments the recursion
$$
m_n'
\ =\ \sum_{k=1}^n {{n-1}\choose{k-1}}\kappa_k(G-\gamma)m_{n-k}'
\ =\ (n-1)!\sum_{k=2}^n \frac{1}{(n-k)!}\zeta(k)m_{n-k}',\qquad n\in\nz,
$$
with solution
\begin{equation}
m_n'\ =\ \sum_{\pi\in{\cal P}_n}\prod_{B\in\pi}\kappa_{|B|}(G-\gamma) \ =\
n!\sum_{a_2,\ldots,a_n}\prod_{i=2}^n\frac{1}{a_i!}\bigg(\frac{\zeta(i)}{i}\bigg)^{a_i},
\qquad n\in\nz,
\end{equation}
where the last sum $\sum_{a_2,\ldots,a_n}$ extends over all
$a_2,\ldots,a_n\in\nz_0$ satisfying $\sum_{i=2}^n ia_i=n$. In
particular, for all $n\ge 2$, $m_n'$ is a polynomial of degree
$\lfloor n/2\rfloor$ in the variable $(\zeta(2),\ldots,\zeta(n))$
with nonnegative integer coefficients. Theorem \ref{cmgumbel}
below provides an alternative formula for $m_n'$.
In order to state the result we
introduce the nonnegative integer coefficients
\begin{equation} \label{dn}
   d_n\ :=\ n!\sum_{j=0}^n \frac{(-1)^j}{j!},\qquad n\in\nz_0.
\end{equation}
Note that $d_n$ is the number of derangements
(fixed point free permutations) of $n$ elements. Clearly,
$d_n=nd_{n-1}+(-1)^n$, $n\in\nz$.
For instance, $d_0=1$, $d_1=0$, $d_2=1$, $d_3=2$, $d_4=9$, $d_5=44$, and
$d_6=265$.
A typical element of ${\cal P}_i$ will be denoted by
$\pi:=\{B_1,\ldots,B_l\}$, where $l:=|\pi|\in\{1,\ldots,i\}$ is the number
of blocks of the partition and the blocks $B_1,\ldots,B_l$
of the partition are non-empty disjoint subsets of $\{1,\ldots,i\}$
satisfying $B_1\cup\cdots\cup B_l=\{1,\ldots,i\}$. Note that the order of
the blocks is unimportant. For a set $B$ and indices $n_b$,
$b\in B$, we use in the following the notation $n_B:=\sum_{b\in B} n_b$.
\begin{theorem}[Alternative formula for the central moments] 
\label{cmgumbel}
   A standard Gumbel distributed random variable $G$ has central moments
   $m_0'=1$, $m_1'=0$ and
   \begin{equation} \label{central}
      m_n' 
      \ =\ n!\sum_{i=1}^n \frac{1}{i!}
           \sum_{{n_1,\ldots,n_i\ge 2}\atop{n_1+\cdots+n_i=n}}
           \frac{d_{n_1}\cdots d_{n_i}}{n_1!\cdots n_i!}\,s_i(n_1,\ldots,n_i),
           \qquad n\ge 2,
   \end{equation}
   with coefficients $d_n$ defined in (\ref{dn}) and
   $s_i(n_1,\ldots,n_i)$ given via
   \begin{eqnarray}
      s_i(n_1,\ldots,n_i)
      & := & \sum_{{k_1,\ldots,k_i\in\nz}\atop{\rm all\ distinct}}
             \frac{1}{k_1^{n_1}\cdots k_i^{n_i}} \label{s1} \\
      & = &  \sum_{l=1}^i (-1)^{i-l} \sum_{\{B_1,\ldots,B_l\}\in{\cal P}_i}
             (|B_1|-1)!\cdots (|B_l|-1)!\,
             \zeta(n_{B_1})\cdots\zeta(n_{B_l}) \label{s2}\\
      & = &
      \sum_{\pi\in{\cal P}_i}(-1)^{i-|\pi|}\prod_{B\in\pi}
      (|B|-1)!\,\zeta(n_B), \label{s3}
   \end{eqnarray}
   where $\zeta$ denotes the zeta function and $|\pi|$ the number of
   blocks of the partition $\pi$.
\end{theorem}
\begin{remarks}
   1. The proof of Theorem \ref{cmgumbel} is based on the fact that
   the random variable $S_n$ defined in (\ref{sn}) converges to $G-\gamma$
   as $n\to\infty$. For more information on this convergence we refer the
   reader to the proof of Theorem \ref{thm3}.

   2. Clearly, $s_i(n_1,\ldots,n_i)$ is symmetric
   with respect to the entries $n_1,\ldots,n_i$. For instance,
   $s_1(n_1)=\zeta(n_1)$, $s_2(n_1,n_2)=\zeta(n_1)\zeta(n_2)-\zeta(n_1+n_2)$
   and $s_3(n_1,n_2,n_3)=\zeta(n_1)\zeta(n_2)\zeta(n_3)-\zeta(n_1)\zeta(n_2+n_3)
   -\zeta(n_2)\zeta(n_1+n_3)-\zeta(n_3)\zeta(n_1+n_2)+2\zeta(n_1+n_2+n_3)$.
   A recursion for $s_i(n_1,\ldots,n_i)$ is provided in (\ref{srec}).
   Eqs.~(\ref{s2}) and (\ref{s3}) are reminiscent of sieve formulas, but we
   have not been able to rigorously relate these equations with some known
   sieve formula.

   3. The values of the central moments $m_n'$ for $0\le n\le 10$ are
   listed in the appendix. The $n$th raw moment $m_n$ of the Gumbel
   distribution is either obtained via (\ref{mnexplicit}) or (\ref{mnsum}),
   or from the central moments $m_j'$, $0\le j\le n$, via
   (\ref{gumbelmoments2}). The book of Srivastava and Choi
   \cite[pp.~370--371]{srivastavachoi2}
   contains the values of $(-1)^nm_n=\Gamma^{(n)}(1)$ for $1\le n\le 10$.
\end{remarks}
\subsection{Proofs} \label{proofs}
\setcounter{theorem}{0}
\begin{proof} (of Proposition \ref{prop1})
   For $n\in\nz$ let $T_n$ denote the absorption time of the Kingman
   coalescent restricted to a sample of size $n$.
   Clearly (see, for example, Kingman \cite[Eq. (5.5)]{kingman1}),
   $T_n\stackrel{d}{=}\sum_{k=2}^n \tau_k$, where $\tau_2,\tau_3,\ldots$ are
   independent random variables and $\tau_k$ is exponentially
   distributed with parameter $\lambda_k=k(k-1)/2$, $k\ge 2$. Thus
   (see, for example, Ross \cite[p.~309]{ross}),
   $T_n$ has a hypoexponential distribution with density $g_n(t):=\sum_{k=2}^n a_{nk}\lambda_k
   e^{-\lambda_kt}$, $t>0$, where
   $$
   a_{nk}\ :=\ \prod_{{j=2}\atop{j\ne k}}^n \frac{\lambda_j}{\lambda_j-\lambda_k}
   \ =\ (-1)^k(2k-1)\frac{n!(n-1)!}{(n-k)!(n+k-1)!},\qquad 2\le k\le n.
   $$
   Alternatively, the density $g_n$ of $T_n$ is obtained as follows.
   Let $D=(D_t)_{t\ge 0}$ denote the block counting process of the
   Kingman coalescent restricted to a sample of size $n$. From
   Lemma \ref{decomposition} (spectral decomposition) provided in the appendix
   it follows that $T_n$ has distribution
   function
   \begin{eqnarray*}
      \pr(T_n\le t)
      & = & \pr(D_t=1)
      \ = \ \sum_{k=1}^n e^{-\lambda_kt} r_{nk}l_{k1}
      \ = \ \sum_{k=1}^n e^{-\lambda_kt}
            \bigg(\prod_{j=k+1}^n\frac{\lambda_j}{\lambda_j-\lambda_k}\bigg)
            \bigg(\prod_{j=1}^{k-1}\frac{\lambda_{j+1}}{\lambda_j-\lambda_k}\bigg)\\
      & = & 1 - \sum_{k=2}^n e^{-\lambda_kt}\prod_{{j=2}\atop{j\ne k}}^n
            \frac{\lambda_j}{\lambda_j-\lambda_k}
      \ = \ 1 - \sum_{k=2}^n a_{nk}e^{-\lambda_kt},
   \qquad t\in [0,\infty),
   \end{eqnarray*}
   where the second last equality holds since $\lambda_1=0$. Taking the derivative with respect to $t$ it follows
   that $T_n$ has density $g_n$. In particular, $T_n$ has moments
   \begin{equation} \label{meantn}
      \me(T_n^j)\ =\ \int_0^\infty t^j g_n(t)\,{\rm d}t
      \ =\ \sum_{k=2}^n a_{nk}\int_0^\infty t^j\lambda_k e^{-\lambda_k t}\,{\rm d}t
      \ =\ j!\sum_{k=2}^n \frac{a_{nk}}{\lambda_k^j},
      \qquad j\in\nz.
   \end{equation}
   In the following it is shown, essentially by letting $n\to\infty$ in
   (\ref{meantn}), that
   \begin{equation} \label{meant}
      \me(T^j)\ =\ j!\sum_{k=2}^\infty \frac{(-1)^k(2k-1)}{\lambda_k^j},
      \qquad j\in\nz.
   \end{equation}
   Eq.~(\ref{meant}) holds for $j=1$, since $\me(T)=2$. Assume now
   that $j\ge 2$.
   From $T_n\nearrow T\stackrel{d}{=}\sum_{k=2}^\infty\tau_k$ almost surely
   as $n\to\infty$ it follows by
   monotone convergence that the left hand side in (\ref{meantn}) converges
   to the left hand side in (\ref{meant}) as $n\to\infty$. In order to see
   that the right hand side in (\ref{meantn}) converges to the right hand side
   in (\ref{meant}) as $n\to\infty$
   fix $\varepsilon>0$. Since $j\ge 2$ and $\lambda_k=k(k-1)/2$,
   the series $\sum_{k=2}^\infty (2k-1)/\lambda_k^j$ is absolutely convergent.
   Thus, there exists a constant $n_0=n_0(\varepsilon)$ such that
   $\sum_{k=n_0+1}^\infty (2k-1)/\lambda_k^j<\varepsilon$. Noting that
   $a_{nk}=(-1)^k(2k-1)b_{nk}$ with
   $$
   0\ \le\ b_{nk}\ :=\ \frac{n!(n-1)!}{(n-k)!(n+k-1)!}
   \ =\ \frac{n(n-1)\cdots(n-k+1)}{(n+k-1)(n+k-2)\cdots n}\ \le\ 1
   $$
   it follows for all $n>n_0$ that
   \begin{eqnarray*}
      \bigg|\sum_{k=2}^n \frac{a_{nk}}{\lambda_k^j}-\sum_{k=2}^\infty \frac{(-1)^k(2k-1)}{\lambda_k^j}\bigg|
      & \le & \sum_{k=2}^n \frac{2k-1}{\lambda_k^j}\underbrace{|b_{nk}-1|}_{\le 1} +
         \sum_{k=n+1}^\infty \frac{2k-1}{\lambda_k^j}\\
      & \le & \sum_{k=2}^{n_0} \frac{2k-1}{\lambda_k^j}|b_{nk}-1| + \sum_{k=n_0+1}^\infty \frac{2k-1}{\lambda_k^j}\\
      & \le & \sum_{k=2}^{n_0}\frac{2k-1}{\lambda_k^j}|b_{nk}-1| + \varepsilon
      \ \to\ \varepsilon
   \end{eqnarray*}
   as $n\to\infty$, since $b_{nk}\to 1$ as $n\to\infty$ for each fixed
   $k\in\{2,3,\ldots\}$. Since $\varepsilon>0$ can be chosen arbitrarily
   small, it follows that the right hand side in (\ref{meantn}) converges
   to the right hand side in (\ref{meant}). Thus, (\ref{meant}) is established.
   Distinguishing in (\ref{meant}) even $k=2m$ and odd $k=2m+1$ and summing over
   all $m\in\nz$ it follows that
   $$
   \me(T^j) 
   \ =\ j!\sum_{m=1}^\infty \frac{1}{m^j}\bigg(\frac{4m-1}{(2m-1)^j}-\frac{4m+1}{(2m+1)^j}\bigg),
   \quad j\in\nz,
   $$
   and (\ref{mean1}) is established.
   Let us now turn to the cumulants of $T$.
   Note that the exponential distribution with parameter
   $\lambda\in (0,\infty)$ has $j$th cumulant $(j-1)!/\lambda^j$.
   From $T_n\stackrel{d}{=}\sum_{k=2}^n\tau_k$
   it follows that $T_n$ has cumulants
   \begin{equation} \label{cum}
   \kappa_j(T_n)\ =\ \sum_{k=2}^n \kappa_j(\tau_k)
   \ =\ \sum_{k=2}^n \frac{(j-1)!}{\lambda_k^j}
   \ =\ (j-1)!2^j\sum_{k=2}^n \frac{1}{k^j(k-1)^j},
   \qquad j\in\nz.
   \end{equation}
   We have verified above that the moments of $T_n$ converge to those of
   $T$, which implies that the cumulants of $T_n$ converge to those of $T$.
   Letting $n\to\infty$ in (\ref{cum}) yields (\ref{cumulant1}).

   Note that $\kappa_j(T)/(j-1)!=\sum_{k=2}^\infty 1/\lambda_k^j\to 1$
   and that $\me(T^j)/j!=\sum_{k=2}^n (-1)^k(2k-1)/\lambda_k^j
   \to 3$ as $j\to\infty$, since $\lambda_2=1$ and $\lambda_k\ge \lambda_3=3$
   for all $k\ge 3$.

   It remains to verify the alternative formula (\ref{meanalt}) for the
   moments of $T$. For all $j\in\nz$ we have
   \begin{eqnarray*}
      \me(T^j)
      & = & \me\bigg(\bigg(\sum_{k=2}^\infty\tau_k\bigg)^j\bigg)
      \ = \ \sum_{k_1,\ldots,k_j=2}^\infty \me(\tau_{k_1}\cdots\tau_{k_j})\\
      & = & \sum_{k_1,\ldots,k_j=2}^\infty \prod_{m=2}^\infty \me(\tau_m^{a_m})
      \ = \ \sum_{2\le k_1\le\cdots\le k_j}\frac{j!}{\prod_{m=2}^\infty a_m!}\prod_{m=2}^\infty\me(\tau_m^{a_m}),
   \end{eqnarray*}
   where, for $m\ge 2$, $a_m$ denotes the number of indices $k_1,\ldots,k_j$
   being equal to $m$. Note that $\sum_{m=2}^\infty a_m=j$. Since
   $\me(\tau_m^{a_m})=a_m!/\lambda_m^{a_m}$, the above
   expression simplifies to
   $$
   \me(T^j)\ =\ j!\sum_{2\le k_1\le\cdots\le k_j}
   \frac{1}{\lambda_{k_1}\cdots\lambda_{k_j}}
   \ =\ j!2^j\sum_{2\le k_1\le\cdots\le k_j}\prod_{i=1}^j\frac{1}{k_i(k_i-1)},
   $$
   which is (\ref{meanalt}).
   \hfill$\Box$
\end{proof}
The proof of Theorem \ref{thm2} is based on the following basic but
fundamental lemma, which provides a solution for a certain two dimensional
recursion with a countable number of initial values.
\begin{lemma} \label{reclemma}
   Let $a_1,b_1,a_2,b_2,\ldots\in\rz$. For $i,j\in\nz$ define $s_{ij}$
   recursively via $s_{ij}:=s_{i-1,j}-s_{i,j-1}$, with initial values
   $s_{0k}:=a_k$ and $s_{k0}:=b_k$, $k\in\nz$. Then
   \begin{equation} \label{solution}
      s_{ij}\ =\ \sum_{k=1}^j (-1)^{j-k}{{i+j-k-1}\choose{i-1}}a_k
      + (-1)^j \sum_{k=1}^i {{i+j-k-1}\choose{j-1}}b_k,
      \qquad i,j\in\nz.
   \end{equation}
\end{lemma}
\begin{proof} (of Lemma \ref{reclemma})
   Induction on $i+j$. Clearly,
   (\ref{solution}) holds for $i+j=2$, 
   since $s_{11}=s_{01}-s_{10}=a_1-b_1$. For the induction step from $i+j-1$
   to $i+j$ ($>2$) three cases are distinguished.

   \vspace{2mm}

   Case 1: If $j=1$, then $i>1$ and, hence, by the recursion and by
   induction,
   $
   s_{i1}=s_{i-1,1}-s_{i0}
   =(a_1-\sum_{k=1}^{i-1}b_k)-b_i
   =a_1-\sum_{k=1}^i b_k$,
   which is (\ref{solution}) for $j=1$.

   \vspace{2mm}

   Case 2: If $i=1$, then $j>1$ and, hence, by the recursion and by
   induction,
   $
   s_{1j}=s_{0j}-s_{1,j-1}
   =a_j-(\sum_{k=1}^{j-1}(-1)^{j-1-k}a_k + (-1)^{j-1}b_1)
   =\sum_{k=1}^j (-1)^{j-k}a_k + (-1)^jb_1$,
   which is (\ref{solution}) for $i=1$.

   \vspace{2mm}

   Case 3: If $i,j>1$ then
   \begin{eqnarray*}
      s_{ij}
      & = & s_{i-1,j} - s_{i,j-1}\\
      & = & \sum_{k=1}^j (-1)^{j-k}{{i+j-k-2}\choose{i-2}}a_k
            + (-1)^j \sum_{k=1}^{i-1}{{i+j-k-2}\choose{j-1}}b_k\\
      &   & \hspace{1cm} - \sum_{k=1}^{j-1} (-1)^{j-1-k}{{i+j-k-2}\choose{i-1}}a_k
            - (-1)^{j-1}\sum_{k=1}^i {{i+j-k-2}\choose{j-2}}b_k\\
      & = & a_j+\sum_{k=1}^{j-1}(-1)^{j-k}\bigg({{i+j-k-2}\choose{i-2}}+{{i+j-k-2}\choose{i-1}}\bigg) a_k\\
      &   & \hspace{1cm} + (-1)^jb_i + (-1)^j\sum_{k=1}^{i-1} \bigg({{i+j-k-2}\choose{j-2}}+{{i+j-k-2}\choose{j-1}}\bigg)b_k\\
      & = & \sum_{k=1}^j (-1)^{j-k}{{i+j-k-1}\choose{i-1}}a_k
            + (-1)^j\sum_{k=1}^i {{i+j-k-1}\choose{j-1}}b_k,
   \end{eqnarray*}
   which completes the induction.\hfill$\Box$
\end{proof}
Before we come to the proof of Theorem \ref{thm2} we provide
a typical application of Lemma \ref{reclemma} showing that this lemma
can be used to determine the value of certain series.
\begin{example}
   For $j\in\nz$ we would like to determine the series
   $\sum_{k=2}^\infty 1/(k^j(k-1)^j)$. We proceed as follows.
   For $i,j\in\nz_0:=\{0,1,2,\ldots\}$ with $i+j\ge 2$ define
   $s_{ij}:=\sum_{k=2}^\infty 1/(k^i(k-1)^j)$. Note that $s_{11}=1$,
   $s_{i0}=\zeta(i)-1$, $i\ge 2$, and that $s_{0j}=\zeta(j)$, $j\ge 2$.
   For all $i,j\in\nz$ with $i+j\ge 3$,
   $$
   s_{ij}\ =\ \sum_{k=2}^\infty \frac{1}{k^{i-1}(k-1)^{j-1}}\frac{1}{k(k-1)}
   \ =\ \sum_{k=2}^\infty \frac{1}{k^{i-1}(k-1)^{j-1}}
   \bigg(\frac{1}{k-1}-\frac{1}{k}\bigg)
   \ =\ s_{i-1,j} - s_{i,j-1}.
   $$
   If we additionally define $s_{01}:=1$ and
   $s_{10}:=0$, then this recursion holds also for $i=j=1$, so for all
   $i,j\in\nz$. By Lemma \ref{reclemma}, for all $j\in\nz$,
   \begin{eqnarray*}
      s_{jj}
      & = & \sum_{k=1}^j (-1)^{j-k}{{2j-k-1}\choose{j-1}}s_{0k} + (-1)^j
            \sum_{k=1}^j {{2j-k-1}\choose{j-1}}s_{k0}\\
      & = & (-1)^j\sum_{{k=2}\atop{k\rm\ even}}^j {{2j-k-1}\choose{j-1}}(s_{0k}+s_{k0})
            + (-1)^j\sum_{{k=1}\atop{k\rm\ odd}}^j {{2j-k-1}\choose{j-1}}(s_{k0}-s_{0k}).
   \end{eqnarray*}
   Plugging in $s_{0k}+s_{k0}=2\zeta(k)-1$ for even $k$ and $s_{k0}-s_{0k}=-1$
   for odd $k$ we obtain the solution
   \begin{eqnarray}
      \sum_{k=2}^\infty \frac{1}{k^j(k-1)^j}
      & = & s_{jj}
      \ = \ (-1)^{j+1}\sum_{k=1}^j {{2j-k-1}\choose{j-1}} +
            2(-1)^j\sum_{{k=2}\atop{k\rm\ even}}^j {{2j-k-1}\choose{j-1}}\zeta(k)\nonumber\\
      & = & (-1)^{j+1}{{2j-1}\choose j} + 2(-1)^j
            \sum_{m=1}^{\lfloor j/2\rfloor} {{2j-2m-1}\choose{j-1}}\zeta(2m)\nonumber\\
      & = & 2(-1)^j\sum_{m=0}^{\lfloor j/2\rfloor}{{2j-2m-1}\choose{j-1}}\zeta(2m),
            \qquad j\in\nz, \label{example}
   \end{eqnarray}
   since $\zeta(0)=-1/2$.
   For instance $s_{22}=2\zeta(2)-3=\pi^2/3-3\approx 0.28987$ and
   $s_{33}=10-6\zeta(2)=10-\pi^2\approx 0.13040$.
\end{example}
The following proof of Theorem \ref{thm2} has much in common with
the previous example, but a modified double sequence
$(s_{ij})_{i,j\in\nz}$ is used having in particular more involved
initial values.
\begin{proof} (of Theorem \ref{thm2})
   The formula (\ref{cumulant2}) for the cumulants of $T$ follows
   directly from (\ref{cumulant1}) by multiplying (\ref{example})
   with $(j-1)!2^j$. In order to verify the formula (\ref{mean2}) for the
   moments of $T$ define for $i,j\in\nz_0$ with $i+j\ge 2$
   $$
   s_{ij}
   \ :=\ \sum_{k=2}^\infty \frac{(-1)^k(2k-1)}{k^i(k-1)^j}.
   $$
   By Proposition \ref{prop1}, $\me(T^j)=j!2^js_{jj}$. Thus it remains
   to verify that
   \begin{equation} \label{sjjsolution}
      s_{jj}\ =\ 2(-1)^j\sum_{m=0}^{\lfloor j/2\rfloor}(2m-1)
      \bigg(1-\frac{1}{2^{2m-1}}\bigg)\frac{(2j-2m-2)!}{(j-1)!(j-2m)!}\zeta(2m),
      \qquad j\in\nz.
   \end{equation}
   It is straightforward to check that
   $s_{11}=1$, $s_{02}=2\log 2+\zeta(2)/2$, 
   $s_{20}=1-2\log 2+\zeta(2)/2$, 
   $$
   s_{i0}\ =\ 1 - 2\bigg(1-\frac{1}{2^{i-2}}\bigg)\zeta(i-1) + \bigg(1-\frac{1}{2^{i-1}}\bigg) \zeta(i),
   \qquad i\ge 3,
   $$
   and
   $$
   s_{0j}\ =\ 2\bigg(1-\frac{1}{2^{j-2}}\bigg)\zeta(j-1) +
   \bigg(1-\frac{1}{2^{j-1}}\bigg)\zeta(j),\qquad j\ge 3.
   $$
   For all $i,j\in\nz$ with $i+j\ge 3$ we have
   \begin{equation} \label{rec}
      s_{ij}
      \ = \ \sum_{k=2}^\infty \frac{(-1)^k(2k-1)}{k^{i-1}(k-1)^{j-1}}
            \frac{1}{k(k-1)}
      \ = \ \sum_{k=2}^\infty \frac{(-1)^k(2k-1)}{k^{i-1}(k-1)^{j-1}}
            \bigg(\frac{1}{k-1}-\frac{1}{k}\bigg)
      \ = \ s_{i-1,j} - s_{i,j-1}.
   \end{equation}
   If we additionally define $s_{01}:=1$
   and $s_{10}:=0$, then the recursion (\ref{rec}) holds also for $i=j=1$, so
   for all $i,j\in\nz$. By Lemma \ref{reclemma},
   $$
   s_{jj}
   \ =\ \sum_{k=1}^j (-1)^{j-k}{{2j-k-1}\choose{j-1}} s_{0k}
        + (-1)^j\sum_{k=1}^j {{2j-k-1}\choose{j-1}}s_{k0},
   \qquad j\in\nz.
   $$
   Ordering with respect to even and odd $k$ yields
   $$
   (-1)^js_{jj}
   \ =\ \sum_{{k=2}\atop{k\rm\ even}}^j {{2j-k-1}\choose{j-1}}(s_{0k}+s_{k0})
        + \sum_{{k=1}\atop{k\rm\ odd}}^j {{2j-k-1}\choose{j-1}}(s_{k0}-s_{0k}).
   $$
   Plugging in $s_{0k}+s_{k0}=1+2(1-1/2^{k-1})\zeta(k)$ for even $k$,
   $s_{10}-s_{01}=-1$, and $s_{k0}-s_{0k}=1-4(1-1/2^{k-2}))\zeta(k-1)$ for
   odd $k\ge 3$, it follows that
   \begin{eqnarray*}
      (-1)^js_{jj}
      & = & \sum_{{k=2}\atop{k\rm\ even}}^j {{2j-k-1}\choose{j-1}}
            \bigg(1+2\bigg(1-\frac{1}{2^{k-1}}\bigg)\zeta(k)\bigg)\\
      &   & \hspace{1cm}
            + {{2j-2}\choose{j-1}}(-1)
            + \sum_{{k=3}\atop{k\rm\ odd}}^j {{2j-k-1}\choose{j-1}}
              \bigg(1-4\bigg(1-\frac{1}{2^{k-2}}\bigg)\zeta(k-1)\bigg)\\
      & = & \sum_{k=2}^j {{2j-k-1}\choose{j-1}} - {{2j-2}\choose{j-1}}
            + 2\sum_{{k=2}\atop{k\rm\ even}}^j {{2j-k-1}\choose{j-1}}\bigg(1-\frac{1}{2^{k-1}}\bigg)\zeta(k)\\
      &   & \hspace{1cm} - 4\sum_{{k=3}\atop{k\rm\ odd}}^j {{2j-k-1}\choose{j-1}}\bigg(1-\frac{1}{2^{k-2}}\bigg)\zeta(k-1)\\
      & = & {{2j-2}\choose j} - {{2j-2}\choose{j-1}}
            + 2\sum_{{k=2}\atop{k\rm\ even}}^j {{2j-k-1}\choose{j-1}}
            \bigg(1-\frac{1}{2^{k-1}}\bigg)\zeta(k)\\
      &   & \hspace{1cm} - 4\sum_{{k=2}\atop{k\rm\ even}}^{j-1} {{2j-k-2}\choose{j-1}}\bigg(1-\frac{1}{2^{k-1}}\bigg)\zeta(k).
   \end{eqnarray*}
   Thus,
   \begin{equation} \label{sjj}
      (-1)^js_{jj}\ =\ c_j + \sum_{{k=2}\atop{k\rm\ even}}^j c_{jk}\zeta(k),
   \end{equation}
   where $c_j:={{2j-2}\choose j}-{{2j-2}\choose{j-1}}=-(2j-2)!/(j!(j-1)!)$ and
   \begin{eqnarray*}
      c_{jk}
      & := & 2\bigg(1-\frac{1}{2^{k-1}}\bigg)
             \bigg({{2j-k-1}\choose{j-1}}-2{{2j-k-2}\choose{j-1}}\bigg)\\
      & = & 2\bigg(1-\frac{1}{2^{k-1}}\bigg)
            \bigg(\frac{(2j-k-1)!}{(j-1)!(j-k)!}-2\frac{(2j-k-2)!}{(j-1)!(j-k-1)!}\bigg)\\
      & = & 2\bigg(1-\frac{1}{2^{k-1}}\bigg)
            \frac{(2j-k-2)!}{(j-1)!(j-k)!}\big((2j-k-1)-2(j-k)\big)\\
      & = & 2(k-1)\bigg(1-\frac{1}{2^{k-1}}\bigg)\frac{(2j-k-2)!}{(j-1)!(j-k)!}.
   \end{eqnarray*}
   Substituting $k=2m$ in (\ref{sjj}) and
   noting that $\zeta(0)=-1/2$ yields (\ref{sjjsolution}). The proof
   is complete.\hfill$\Box$
\end{proof}
In order to prepare the proof of Theorem \ref{thm3} recall that a standard
Gumbel distributed random variable $G$ has
(see, for example, \cite[p.~12, Eqs.~(22.29) and (22.30)]{johnsonkotz})
cumulants $\kappa_1(G)=\gamma$ (Euler's constant) and
$\kappa_j(G)=(-1)^j\Psi^{(j-1)}(1)=(j-1)!\,\zeta(j)$, $j\ge 2$.
\begin{proof}
   (of Theorem \ref{thm3})
   We start similar as in the proof of Proposition \ref{prop1}.
   Clearly, $L_n\stackrel{d}{=}\sum_{k=2}^n k\tau_k=\sum_{k=2}^n X_k$, where
   the random variables $X_k:=k\tau_k$, $k\ge 2$, are independent and $X_k$ is
   exponentially distributed with parameter $\mu_k:=\lambda_k/k=(k-1)/2$, $k\ge 2$.
   Thus, $L_n$ has cumulants
   \begin{equation} \label{lncum}
   \kappa_j(L_n)
   \ =\ \sum_{k=2}^n \kappa_j(X_k)
   \ =\ \sum_{k=2}^n \frac{(j-1)!}{\mu_k^j}
   \ =\ (j-1)!2^j\sum_{k=2}^n \frac{1}{(k-1)^j},\qquad j\in\nz.
   \end{equation}
   For $j=1$ we have $\kappa_1(G_n)=\kappa_1(L_n/2-\log n)=
   \kappa_1(L_n)/2-\log n=\sum_{k=2}^n 1/(k-1)-\log n\to\gamma
   =\kappa_1(G)$ as $n\to\infty$. For $j\ge 2$ we have
   $\kappa_j(G_n)=\kappa_j(L_n)/2^j=(j-1)!\sum_{k=2}^n 1/(k-1)^j
   \to (j-1)!\zeta(j)=\kappa_j(G)$ as $n\to\infty$.
   Thus, we have convergence $\kappa_j(G_n)\to\kappa_j(G)$ as $n\to\infty$
   of all cumulants, which implies the
   convergence $\me(G_n^j)\to\me(G^j)$ as $n\to\infty$ of all moments.

   For the convergence $G_n:=L_n/2-\log n\to G$ in distribution as $n\to\infty$ we refer
   the reader to \cite[Lemma 7.1]{drmotaiksanovmoehleroesler} and the references
   in the remark thereafter.
   In order to verify that the convergence $G_n\to G$ holds even almost
   surely define
   \begin{equation} \label{sn}
      S_n\ :=\ \frac{L_n-\me(L_n)}{2}
      \ =\ \sum_{k=2}^n \frac{X_k-\me(X_k)}{2}
      \ =\ \sum_{k=2}^n Y_k,\qquad n\in\nz,
   \end{equation}
   where $Y_k:=(X_k-\me(X_k))/2$, $k\ge 2$. Note that $Y_2,Y_3,\ldots$
   are independent. From $\me(L_n)=2\sum_{k=2}^n 1/(k-1)=
   2\log n +2\gamma+O(1/n)$ we conclude that $S_n\to G-\gamma$ in
   distribution as $n\to\infty$. It is well known (see, for example
   \cite[Theorem 22.7]{billingsley}) that a sum
   $S_n=\sum_{k=2}^n Y_k$ of independent random variables converges
   in distribution if and only if it converges almost surely. Thus
   we even have $S_n\to G-\gamma$ almost surely. Thus,
   $G_n\to G$ almost surely as $n\to\infty$. By \cite[Proposition 3.12]{kallenberg}
   it follows that $G_n\to G$ in $L^p$ for any $p\in (0,\infty)$,
   and the sequence $(G_n^p)_{n\in\nz}$ is uniformly integrable for any
   $p\in (0,\infty)$. Note that $\|S_n-(G-\gamma)\|_p
   \le\|\gamma+S_n-G_n\|_p+\|G_n-G\|_p=\|\gamma+\log n-\sum_{k=2}^n 1/(k-1)\|_p+
   \|G_n-G\|_p\to 0$ as $n\to\infty$, so we also have $S_n\to G-\gamma$
   in $L^p$ for any $p\in (0,\infty)$ and, hence, convergence
   $\me(S_n^j)\to \me((G-\gamma)^j)$ as $n\to\infty$ of all moments. Furthermore,
   $(S_n)_{n\in\nz}$ is a martingale, but we did
   not use this property in the proof.
   \hfill$\Box$
\end{proof}
\begin{remark}
   In this remark formulas for the moments of the total
   tree length $L_n$ are provided. It is known (see, for example,
\cite[Lemma 7.1]{drmotaiksanovmoehleroesler} and the remark thereafter)
that $L_n$ has the same distribution as the maximum of $n-1$ independent
and exponentially distributed random variables
with parameter $1/2$.
In particular, $L_n$ has distribution function $\pr(L_n\le t)
=(1-e^{-t/2})^{n-1}$, $t>0$, and, hence, moments
\begin{eqnarray}
   \me(L_n^j)
   & = & \int_0^\infty jt^{j-1}\pr(L_n>t)\,{\rm d}t
   \ = \ \int_0^\infty jt^{j-1}\big(1-(1-e^{-t/2})^{n-1}\big)\,{\rm d}t\nonumber\\
   & = & \int_0^\infty -jt^{j-1}\sum_{k=1}^{n-1}{{n-1}\choose k}(-e^{-t/2})^k\,{\rm d}t
   \ = \ \sum_{k=1}^{n-1} (-1)^{k+1}{{n-1}\choose k}\int_0^\infty jt^{j-1}e^{-kt/2}\,{\rm d}t\nonumber\\
   & = & \sum_{k=1}^{n-1} (-1)^{k+1}{{n-1}\choose k}\frac{j!}{(k/2)^j}
   \ = \ j!2^j\sum_{k=1}^{n-1} \frac{(-1)^{k+1}}{k^j}{{n-1}\choose k},
   \qquad j\in\nz. \label{lnmom}
\end{eqnarray}
   Alternatively,
   \begin{eqnarray*}
      \me(L_n^j)
      & = & \me\bigg(\bigg(\sum_{k=2}^n X_k\bigg)^j\bigg)
      \ = \ \sum_{k_1,\ldots,k_j=2}^n \me(X_{k_1}\cdots X_{k_j})\\
      & = & \sum_{k_1,\ldots,k_j=2}^n \me(X_2^{a_2})\cdots\me(X_n^{a_n})
      \ = \ \sum_{2\le k_1\le\cdots\le k_j\le n}\frac{j!}{a_2!\cdots a_n!}\,
            \me(X_2^{a_2})\cdots\me(X_n^{a_n}) ,
   \end{eqnarray*}
   where, for $m\in\{2,\ldots,n\}$, $a_m$ denotes the number of indices
   $k_1,\ldots,k_j$ being equal to $m$. Since $\me(X_m^{a_m})=a_m!/\mu_k^{a_m}$,
   the above expression simplifies to
   \begin{equation} \label{lnmom2}
      \me(L_n^j)
      \ =\ j!\sum_{2\le k_1\le\cdots\le k_j\le n}\frac{1}{\mu_{k_1}\cdots\mu_{k_j}}
      \ =\ j!2^j
           \sum_{1\le k_1\le\cdots\le k_j\le n-1} \frac{1}{k_1\cdots k_j},
      \quad j\in\nz.
   \end{equation}
   Comparing (\ref{lnmom}) with (\ref{lnmom2}) leads to the combinatorial
   identity
   $$
   \sum_{k=1}^{n-1} \frac{(-1)^{k+1}}{k^j}{{n-1}\choose k}
   \ =\ \sum_{1\le k_1\le\cdots\le k_j\le n-1}\frac{1}{k_1\cdots k_j},
   \qquad j\in\nz, n\ge 2.
   $$
\end{remark}
\begin{proof} (of Theorem \ref{cmgumbel})
   Note first that an exponentially distributed random variable $X$ with
   parameter $\alpha\in (0,\infty)$ has moments $\me(X^n)=n!/\alpha^n$
   and central moments $\me((X-\me(X))^n))=d_n/\alpha^n$
   with $d_n$ defined in (\ref{dn}), $n\in\nz_0$. Consider the random variable
   $S_N$ defined via (\ref{sn}).
   From the proof of Theorem \ref{thm3} it is already known
   that $S_N\to G-\gamma$ almost surely as $N\to\infty$
   with convergence of all moments. Moreover, for all $n\in\nz_0$,
   \begin{eqnarray*}
      \me(S_N^n)
      & = & \me((Y_2+\cdots+Y_N)^n)\\
      & = &  \sum_{i=1}^n
             \sum_{2\le k_1<\cdots<k_i\le N}
             \sum_{{n_1,\ldots,n_i\ge 1}\atop{n_1+\cdots+n_i=n}}
             \frac{n!}{n_1!\cdots n_i!}\me(Y_{k_1}^{n_1})\cdots\me(Y_{k_i}^{n_i})\\
      & = & \sum_{i=1}^n \frac{1}{i!}
            \sum_{{k_1,\ldots,k_i=2}\atop{\rm all\ distinct}}^N
            \sum_{{n_1,\ldots,n_i\ge 1}\atop{n_1+\cdots+n_i=n}}
            \frac{n!}{n_1!\cdots n_i!}
            \me(Y_{k_1}^{n_1})\cdots\me(Y_{k_i}^{n_i})\\
      & = & n!\sum_{i=1}^n \frac{1}{i!}
            \sum_{{n_1,\ldots,n_i\ge 1}\atop{n_1+\cdots+n_i=n}}
            \frac{1}{n_1!\cdots n_i!}
            \sum_{{k_1,\ldots,k_i=2}\atop{\rm all\ distinct}}^N
            \frac{d_{n_1}}{(k_1-1)^{n_1}}\cdots\frac{d_{n_i}}{(k_i-1)^{n_i}}\\
      & = & n!\sum_{i=1}^n \frac{1}{i!}
            \sum_{{n_1,\ldots,n_i\ge 2}\atop{n_1+\cdots+n_i=n}}
            \frac{d_{n_1}\cdots d_{n_i}}{n_1!\cdots n_i!}
            \sum_{{k_1,\ldots,k_i=1}\atop{\rm all\ distinct}}^{N-1}
            \frac{1}{k_1^{n_1}\cdots k_i^{n_i}},
   \end{eqnarray*}
   since $d_1=0$. Letting $N\to\infty$ it follows that $G$ has central moments
   (\ref{central}) with $s_i(n_1,\ldots,n_i)$ defined via (\ref{s1}).
   Obviously, the expressions in (\ref{s2}) and (\ref{s3}) coincide.
   Thus, it remains to verify that $s_i(n_1,\ldots,n_i)$ can be expressed in
   terms of the zeta function via (\ref{s2}). We show this by induction on $i\in\nz$.
   Clearly, (\ref{s2}) holds for $i=1$, since $s_1(n_1)=\zeta(n_1)$ for all
   $n_1\ge 2$.
   Concerning the induction step from $1,\ldots,i$ to $i+1$ ($\ge 2$) note
   first that, for all $i\in\nz$ and all $n_1,\ldots,n_i\ge 2$,
   \begin{eqnarray}
      &   & \hspace{-25mm}
            s_{i+1}(n_1,\ldots,n_{i+1})
      \ = \ \sum_{{k_1,\ldots,k_i\in\nz}\atop{\rm all\ distinct}}
            \frac{1}{k_1^{n_1}\cdots k_i^{n_i}}
            \sum_{k_{i+1}\in\nz\setminus\{k_1,\ldots,k_i\}}
            \frac{1}{k_{i+1}^{n_{i+1}}}\nonumber\\
      & = & \sum_{{k_1,\ldots,k_i\in\nz}\atop{\rm all\ distinct}}
            \frac{1}{k_1^{n_1}\cdots k_i^{n_i}}
            \bigg(
               \sum_{k_{i+1}\in\nz}\frac{1}{k_{i+1}^{n_{i+1}}}
               - \sum_{r=1}^i \frac{1}{k_r^{n_{i+1}}}
            \bigg)\nonumber\\
      & = & s_i(n_1,\ldots,n_i)\,s_1(n_{i+1}) - \sum_{r=1}^i s_i(n_1,\ldots,n_{r-1},n_r+n_{i+1},n_{r+1},\ldots,n_i).
      \label{srec}
   \end{eqnarray}
   By induction we conclude that
   \begin{eqnarray}
      &   & \hspace{-15mm}
            s_i(n_1,\ldots,n_i)\,s_1(n_{i+1})\nonumber\\
      & = & \sum_{l=1}^i (-1)^{i-l}
            \sum_{\{B_1,\ldots,B_l\}\in{\cal P}_i}
            (|B_1|-1)!\cdots(|B_l|-1)!\,
            \zeta(n_{B_1})\cdots\zeta(n_{B_l})\zeta(n_{i+1})\nonumber\\
      & = & \sum_{l=1}^i (-1)^{i-l}
            \sum_{{\{B_1,\ldots,B_{l+1}\}\in{\cal P}_{i+1}}\atop{B_{l+1}=\{i+1\}}}
            (|B_1|-1)!\cdots (|B_{l+1}|-1)!\,
            \zeta(n_{B_1})\cdots\zeta(n_{B_{l+1}})\nonumber\\
      & = & \sum_{l=2}^{i+1} (-1)^{i+1-l}
            \sum_{{\{B_1,\ldots,B_l\}\in{\cal P}_{i+1}}\atop{B_l=\{i+1\}}}
            (|B_1|-1)!\cdots (|B_l|-1)!\,
            \zeta(n_{B_1})\cdots\zeta(n_{B_l}).\label{induction1}
   \end{eqnarray}
   Also by induction it is seen that
   \begin{eqnarray*}
      &   & \hspace{-15mm}\sum_{r=1}^i s_i(n_1,\ldots,n_{r-1},n_r+n_{i+1},n_{r+1},\ldots,n_i)\\
      & = & \sum_{r=1}^i \sum_{l=1}^i (-1)^{i-l}
            \sum_{{\{A_1,\ldots,A_l\}\in{\cal P}_i}\atop{r\in A_l}}
            (|A_1|-1)!\cdots (|A_l|-1)!\,
            \zeta(n_{A_1})\cdots\zeta(n_{A_{l-1}})\zeta(n_{A_l\cup\{i+1\}}),
   \end{eqnarray*}
   where, in the last sum, we ordered without loss of generality the blocks $A_1,\ldots,A_l$ of the
   partition such that the element $r$ belongs to the last block $A_l$.
   Reordering the sums on the right hand side yields
   \begin{eqnarray*}
      &   & \hspace{-10mm}
            \sum_{r=1}^i s_i(n_1,\ldots,n_{r-1},n_r+n_{i+1},n_{r+1},\ldots,n_i)\\
      & = & \sum_{l=1}^i (-1)^{i-l}
            \sum_{\{A_1,\ldots,A_l\}\in{\cal P}_i}
            \sum_{{r=1}\atop{r\in A_l}}^i
            (|A_1|-1)!\cdots (|A_l|-1)!\,
            \zeta(n_{A_1})\cdots\zeta(n_{A_{l-1}})\zeta(n_{A_l\cup\{i+1\}}).
   \end{eqnarray*}
   Rewriting this expression in terms of the blocks
   $B_1:=A_1,\ldots,B_{l-1}:=A_{l-1}$ and $B_l:=A_l\cup\{i+1\}$ it follows that
   \begin{eqnarray*}
      &   & \hspace{-10mm}\sum_{r=1}^i
            s_i(n_1,\ldots,n_{r-1},n_r+n_{i+1},n_{r+1},\ldots,n_i)\\
      & = & \sum_{l=1}^i (-1)^{i-l}
            \sum_{\{B_1,\ldots,B_l\}\in{\cal P}_{i+1}}
            \sum_{{r=1}\atop{r,i+1\in B_l}}^i
            (|B_1|-1)!\cdots (|B_{l-1}|-1)!\,(|B_l|-2)!\,
            \zeta(n_{B_1})\cdots\zeta(n_{B_l}).
   \end{eqnarray*}
The last sum (over $r$) consists of $|B_l|-1$ summands and these
summands do not depend on $r$, which gives rise to a factor
$|B_l|-1$ leading to
   \begin{eqnarray}
      &   & \hspace{-15mm}\sum_{r=1}^i s_i(n_1,\ldots,n_{r-1},n_r+n_{i+1},n_{r+1},\ldots,n_i)\nonumber\\
      & = & \sum_{l=1}^i {(-1)^{i-l}}
            \sum_{{\{B_1,\ldots,B_l\}\in{\cal P}_{i+1}}\atop{B_k\neq\{i+1\}{\rm\;for\;all\;}k\in\{1,\ldots,l\}}}
            (|B_1|-1)!\cdots (|B_l|-1)!\,
            \zeta(n_{B_1})\cdots\zeta(n_{B_l}).
            \label{induction2}
   \end{eqnarray}
   Subtracting (\ref{induction2}) from (\ref{induction1}) and recalling
   (\ref{srec}) it follows that
   $$
   s_{i+1}(n_1,\ldots,n_{i+1})
   \ =\ \sum_{l=1}^{i+1}(-1)^{i+1-l}\sum_{\{B_1,\ldots,B_l\}\in{\cal P}_{i+1}}
            (|B_1|-1)!\cdots(|B_l|-1)!\,\zeta(n_{B_1})\cdots\zeta(n_{B_l}),
   $$
   which completes the induction. Thus, (\ref{s2}) is established.
   \hfill$\Box$
\end{proof}
\subsection{Appendix}
\subsubsection*{Central moments of the Gumbel distribution}
\setcounter{theorem}{0}
   For completeness we record the central moment $m_n'$ of the
   Gumbel distribution for $0\le n\le 10$. Based on (\ref{central})
   we provide them in terms of the coefficients $s_i(n_1,\ldots,n_i)$,
   in terms of the zeta function, and numerically.
   The first central moments of the Gumbel distribution are
   $m_0'=1$, $m_1'=0$,
   $m_2'=s_1(2)=\zeta(2)=\pi^2/6\approx 1.64493$,
   $m_3'=2s_1(3)=2\zeta(3)\approx 2.40411$,
   $m_4'=9s_1(4)+3s_2(2,2)=6\zeta(4)+3\zeta^2(2)=3\pi^4/20\approx 14.61136$,
   $m_5'=44s_1(5)+20s_2(2,3)=24\zeta(5)+
   20\zeta(2)\zeta(3)=24\zeta(5)+(10/3)\pi^2\zeta(3)
   \approx 64.43235$,
   $$
   \begin{array}{lcl}
   m_6'
   & = & 265s_1(6)+135s_2(2,4)+40s_2(3,3)+15s_3(2,2,2)\vspace{2mm}\\
   & = & 120\zeta(6)+90\zeta(2)\zeta(4)+40\zeta^2(3)+15\zeta^3(2)\vspace{2mm}\\
   & = & \displaystyle\frac{61}{168}\pi^6+40\zeta^2(3)
   \ \approx\ 406.87347,\vspace{3mm}\\
   \end{array}
   $$
   $$
   \begin{array}{lcl}
   m_7'
   & = & 1854s_1(7)+924s_2(2,5)+630s_2(3,4)+210s_3(2,2,3)\vspace{2mm}\\
   & = & 720\zeta(7)+504\zeta(2)\zeta(5)+420\zeta(3)\zeta(4)+210\zeta^2(2)\zeta(3)\vspace{2mm}\\
   & = & \displaystyle 720\zeta(7)+84\pi^2\zeta(5)+\frac{21}{2}\pi^4\zeta(3)
   \ \approx\ 2\,815.13142,\vspace{3mm}\\
   \end{array}
   $$
   $$
   \begin{array}{lcl}
   m_8'
   & = & 14833s_1(8)+7420s_2(2,6)+4928s_2(3,5)+2835s_2(4,4)\vspace{2mm}\\
   &   & \hspace{3cm}+\ 1890s_3(2,2,4)+1120s_3(2,3,3)+105s_4(2,2,2,2)\vspace{2mm}\\
   & = & 5040\zeta(8)+3360\zeta(2)\zeta(6)+2688\zeta(3)\zeta(5)+1260\zeta^2(4)\vspace{2mm}\\
   &   & \hspace{3cm}+\ 1260\zeta^2(2)\zeta(4)+1120\zeta(2)\zeta^2(3)+105\zeta^4(2)\vspace{2mm}\\
   & = & \displaystyle\frac{1261}{720}\pi^8 + 2688\zeta(3)\zeta(5)+\frac{560}{3}
   \pi^2\zeta^2(3)
   \ \approx\ 22630.60731,\vspace{3mm}\\
   \end{array}
   $$
   $$
   \begin{array}{lcl}
   m_9'
   & = & 133496s_1(9)+66744s_2(2,7)+44520s_2(3,6)+49896s_2(4,5)\vspace{2mm}\\
   &   & \hspace{1cm}+\ 16632s_3(2,2,5)+22680s_3(2,3,4)+2240s_3(3,3,3)+2520s_4(2,2,2,3)\vspace{2mm}\\
   & = & 40320\zeta(9)+25920\zeta(2)\zeta(7)+20160\zeta(3)\zeta(6)+18144\zeta(4)\zeta(5)\vspace{2mm}\\
   &   & \hspace{1cm}+\ 9072\zeta^2(2)\zeta(5)+15120\zeta(2)\zeta(3)\zeta(4)+2240\zeta^3(3)+2520\zeta^3(2)\zeta(3)\vspace{2mm}\\
   & = & \displaystyle 40320\zeta(9)+4320\pi^2\zeta(7)+\frac{2268}{5}\pi^4\zeta(5)
         +2240\zeta^3(3)+61\pi^6\zeta(3)\vspace{2mm}\\
   & \approx & 203595.03670,\quad\mbox{and}\vspace{3mm}\\
   \end{array}
   $$
   $$
   \begin{array}{lcl}
   m_{10}'
   & = & 1334961s_1(10)+667485s_2(2,8)+444960s_2(3,7)+500850s_2(4,6)\vspace{2mm}\\
   &   & \hspace{8mm}+\ 243936s_2(5,5)+166950s_3(2,2,6)+221760s_3(2,3,5)+127575s_3(2,4,4)\vspace{2mm}\\
   &   & \hspace{8mm}+\ 75600s_3(3,3,4)+28350s_4(2,2,2,4)+25200s_4(2,2,3,3)+945s_5(2,2,2,2,2)\vspace{2mm}\\
   & = & 362880\zeta(10)+226800\zeta(2)\zeta(8)+172800\zeta(3)\zeta(7)+151200\zeta(4)\zeta(6)\vspace{2mm}\\
   &   & \hspace{8mm}+\ 72576\zeta^2(5)+75600\zeta^2(2)\zeta(6)+120960\zeta(2)\zeta(3)\zeta(5)+56700\zeta(2)\zeta^2(4)\vspace{2mm}\\
   &   & \hspace{8mm}+\ 50400\zeta^2(3)\zeta(4)
         +18900\zeta^3(2)\zeta(4)+25200\zeta^2(2)\zeta^2(3)+945\zeta^5(2)\vspace{2mm}\\
   & = & \displaystyle \frac{4977}{352}\pi^{10}+172800\zeta(3)\zeta(7)
         +72576\zeta^2(5)+20160\pi^2\zeta(3)\zeta(5)+1260\pi^4\zeta^2(3)\vspace{2mm}\\
   & \approx & 2036946.09776.
   \end{array}
   $$
\subsubsection*{A spectral decomposition}
\setcounter{theorem}{0}
\begin{lemma}[Spectral decomposition of a pure death process] \label{decomposition}
   Let $n\in\nz$ and let $X:=(X_t)_{t\ge 0}$ be a pure death process on a
   probability space $(\Omega,{\cal F},\pr)$ with state space
   $\{1,\ldots,n\}$ and pairwise distinct
   death rates $d_1,\ldots,d_n$. Then, the transition probabilities
   $p_{ij}(t):=\pr(X_t=j\,|\,X_0=i)$ are given by
   \begin{equation}
      p_{ij}(t)\ =\ \sum_{k=j}^i e^{-d_kt} r_{ik}l_{kj},
      \qquad i,j\in\{1,\ldots,n\},
   \end{equation}
   where the $n\times n$ matrices $R=(r_{ij})$ and $L=(l_{ij})$ are
   defined via $r_{ij}:=l_{ij}:=0$ for $i<j$ and
   $$
   r_{ij}\ :=\ \prod_{l=j+1}^i \frac{d_l}{d_l-d_j}
   \qquad\mbox{and} \qquad
   l_{ij}\ :=\ \prod_{l=j}^{i-1} \frac{d_{l+1}}{d_l-d_i}
   \qquad\mbox{for $i\ge j$.}
   $$
\end{lemma}
\begin{proof}
   With some effort it can be checked that $RL=I$, where
   $I=(\delta_{ij})_{1\le i,j\le n}$ denotes the $n\times n$ unit matrix.
   Let $Q=(q_{ij})_{1\le i,j\le n}$ denote the generator matrix of $X$,
   i.e. $q_{ii}:=-d_i$, $q_{i,i-1}:=d_i$ and $q_{ij}=0$ otherwise.
   Furthermore, let $D$ denote the diagonal matrix with entries
   $d_{ij}:=-d_i$ for $i=j$ and $d_{ij}:=0$ otherwise.
   It is readily checked that $RD=QR$, and, hence,
   $RDL=Q$. The transition probabilities $p_{ij}(t)$ of the process
   $X$ are now obtained from the spectral decomposition
   $P(t):=e^{tQ}=e^{tRDL}=R(e^{tD})L$ of the transition matrix
   $P(t)$ as $p_{ij}(t)=\sum_{k=1}^n e^{-d_kt} r_{ik} l_{kj}
   =\sum_{k=j}^i e^{-d_kt}r_{ik}l_{kj}$.\hfill$\Box$
\end{proof}

\end{document}